\documentclass[10pt,a4paper]{article}
\usepackage[a4paper,hmargin={1.5cm,1.5cm},vmargin={2cm,2cm}]{geometry}
\usepackage[utf8]{inputenc}
\usepackage[affil-sl]{authblk}

\usepackage[backend=biber,style=ieee,citestyle=numeric-comp,sorting=nty]{biblatex}
\usepackage[dvipsnames]{xcolor}
\usepackage[singlelinecheck=false]{caption}
\usepackage[linesnumbered,ruled]{algorithm2e}
\usepackage[colorlinks=true,linkcolor=Rhodamine,citecolor=Rhodamine,urlcolor=Rhodamine]{hyperref}

\usepackage{url}
\usepackage{hyphenat}
\usepackage{parskip}
\usepackage{tabularray}
\usepackage{amsmath}
\usepackage{upgreek}
\usepackage{breqn}
\usepackage{mathtools}
\usepackage{amsfonts}
\usepackage{amscd}
\usepackage{amssymb}
\usepackage{tikz}
\usepackage{graphicx}
\usepackage{pifont}
\usepackage{mathrsfs}
\usepackage{fancyhdr}
\usepackage{subfiles}

\usepackage{amsthm}
\theoremstyle{definition}
\newtheorem{theorem}{Theorem}
\newtheorem{definition}[theorem]{Definition}
\newtheorem{remark}[theorem]{Remark}

\newtheorem{lemma}[theorem]{Lemma}
\newtheorem{proposition}[theorem]{Proposition}

\definecolor{primary}{HTML}{6200ea}
\definecolor{secondary}{HTML}{00e5ff}

\newcommand{\Iset}{\mathcal{I}}
\newcommand{\Jset}{\mathcal{J}}
\newcommand{\R}{\mathbb{R}}
\newcommand{\B}{\{0,1\}}
\newcommand{\ph}{\varphi}
\newcommand{\BFMWSAT}{\textsc{BF-MWSAT}}

\begin{document}

\font\titleFont=cmr12 at 17pt
\title{{\titleFont Bipolar Fuzzy Minimum-Weight Satisfiability under Continuous t-Norms: A Clause-Witness Branch-and-Bound Framework}}

\def\correspondingauthor{\footnote{Corresponding Author: a.ghodousian@ut.ac.ir (A. Ghodousian)}}

\author[1]{Amin Ghodousian \correspondingauthor{}}
\author[2]{Mohammad Sedigh Chopannavaz}

\affil[1]{Faculty of Engineering Science, College of Engineering, University of Tehran, P.O.Box 11365-4563, Tehran, Iran.}
\affil[2]{Department of Engineering Science, College of Engineering, University of Tehran, Tehran, Iran.}

\date{}

% Bold Proof
\makeatletter
\renewenvironment{proof}[1][\proofname]{%
    \par\pushQED{\qed}\normalfont%
    \topsep6\p@\@plus6\p@\relax
    \trivlist\item[\hskip\labelsep\bfseries#1\@addpunct{.}]%
    \ignorespaces
}{%
    \popQED\endtrivlist\@endpefalse
}
\makeatother

\maketitle

% ==================================== Abstract ====================================
\section*{Abstract}\label{sec_abs}
Minimum-weight satisfiability asks for a satisfying Boolean assignment of minimum cost, while many fuzzy decision and knowledge systems operate with graded positive and negative relations rather than crisp literals. This paper introduces a bipolar fuzzy minimum-weight satisfiability problem in which positive and negative clause--variable relations, target clause levels, and variable states are represented in the unit interval. The clause requirements are expressed by bipolar fuzzy relational equalities defined with an arbitrary continuous t-norm. Classical minimum-weight CNF satisfiability is recovered as a crisp special case, independently of the selected t-norm, and an optimal Boolean solution is shown to exist in that specialization even though the proposed formulation uses continuous variables. For the graded problem, clause satisfaction is characterized through admissible variable domains and effective clause-witness sets. This structure leads to an exact clause-witness branch-and-bound method with preprocessing, propagation, lower-point closure, and a disjoint-witness packing bound. Computational experiments verify the solver against explicit witness enumeration, mixed-integer baselines, and external SATLIB formulas; on a focused graded ablation set, the packing bound reduces the median number of explored nodes by about 45 percent relative to the single-row witness bound.

\textbf{Keywords}: Bipolar fuzzy relational equations, minimum-weight satisfiability, SAT, continuous t-norm, branch-and-bound.

% ==================================== Introduction ====================================
\section{Introduction}\label{sec_intro}
The Boolean satisfiability problem (SAT) is one of the central problems of theoretical computer science and combinatorial optimization. Its importance follows both from its foundational complexity status and from its role as a general modelling language for discrete decision problems. Cook's theorem established the central position of SAT in NP-completeness theory \cite{cook1971}, while the Davis--Putnam and Davis--Logemann--Loveland procedures laid the algorithmic foundations for complete SAT search \cite{davis1960,davis1962}. Modern SAT solving has subsequently developed into a mature computational field through conflict-driven clause learning, efficient propagation, restart policies, learned-clause management, and implementation-level engineering \cite{marquessilva1999,moskewicz2001,marquessilva2021}. Comprehensive treatments of Boolean functions and satisfiability describe how a wide range of combinatorial models can be represented in propositional form and solved by exploiting this structure \cite{crama2011,knuth2015}.

Optimization variants of satisfiability arise when feasibility alone is insufficient. In minimum-weight satisfiability, a nonnegative cost is associated with setting each variable to true and the objective is to find a satisfying assignment of minimum total cost. The unit-weight version is commonly known as Min-Ones SAT, and its restricted forms are closely connected with familiar covering problems; for example, Min-Ones 2-SAT has a strong structural relation with vertex cover \cite{misra2013}. MaxSAT and weighted MaxSAT pursue a related but different goal by maximizing the amount of satisfied Boolean information, and a substantial literature has developed exact, SAT-based, core-guided, and hybrid algorithms for these optimization problems \cite{ansotegui2013,morgado2013,cai2020}. These developments illustrate an important lesson for algorithm design: a useful formulation should expose the combinatorial structure of the problem rather than merely pass the model to a generic optimizer. SAT preprocessing provides a particularly clear example. Unit propagation, clause elimination, implication-based simplification, and related techniques can shrink an instance substantially before the main search begins \cite{jarvisalo2010,heule2011,biere2021preprocessing}.

Classical SAT assumes Boolean truth values and crisp literal occurrences. Fuzzy satisfiability studies relax this setting by allowing graded truth values and fuzzy logical semantics. Different lines of work have been developed according to the underlying fuzzy logic and the intended optimization criterion. El Halaby and Abdalla extended MaxSAT to Lukasiewicz logic and proposed disjunctive-linear, mixed-integer linear, and weighted-constraint encodings \cite{halaby2016}. Recent work has further examined the computational complexity of maximal, weighted, and partial weighted satisfiability in Lukasiewicz logic \cite{lapenta2026}, while mixed-integer nonlinear formulations have been proposed for solving satisfiability under several fuzzy-logical semantics \cite{castro2026}. These models are valuable when the connectives of a fuzzy propositional logic provide the intended semantics. The present paper takes a different route: it starts from a bipolar fuzzy relational system and treats its rows as graded clause requirements. Thus, the proposed model is not intended to replace the established semantics of fuzzy propositional logics; it defines a satisfiability-type optimization problem whose native structure is a system of bipolar fuzzy relational equalities.

Fuzzy relational equations (FREs) provide a well-developed mathematical framework for inverse relational reasoning. Sanchez introduced composite fuzzy relation equations and their use in fuzzy modelling \cite{sanchez1976}; subsequent monographs and surveys established their resolution theory, structural properties, and role in knowledge engineering \cite{dinola1989,dinola1991,lifang2009}. Optimization over FRE solution sets has also received sustained attention. Linear objectives under max--min, max--product, and general max--t-norm compositions have been studied through structural, integer-programming, and analytical approaches \cite{fangli1999,loetamonphong2001,guuwu2002,guuwu2010}. Connections between fuzzy relational equations and covering problems further show that relational equations can encode recognizable combinatorial structures rather than only fuzzy-control calculations \cite{markovskii2005}.

Bipolar fuzzy relational equations (BFREs) extend this framework by allowing a variable and its complement to participate simultaneously. This feature is particularly relevant to clause systems, because a Boolean clause may contain both positive literals and negated literals. Linear optimization with bipolar max--min constraints was investigated by Freson, De Baets, and De Meyer \cite{freson2013}; related optimization and resolution results were subsequently obtained for Lukasiewicz, Hamacher, product, strict, Archimedean, and more general continuous t-norm settings \cite{liliu2014,liuluwu2016,aliannezhadi2016,cornejo2017,cornejo2021,ghodousian2023,ghodousian2025arch}. The generalized continuous-t-norm framework in \cite{ghodousian2025core} characterizes feasibility through componentwise admissibility and activation sets, represents the complete feasible region as a finite union of compact regions, supplies simplification rules, and identifies a finite family containing a global optimum for coordinatewise monotone objectives.

The most direct previous connection between BFREs and SAT is due to Li and Jin \cite{lijin2016}. For bipolar max--min equations, they revealed a polynomial reduction from classical Boolean satisfiability to consistency of a bipolar equation system and proved NP-completeness of the consistency problem. Their construction uses crisp positive and negative incidence matrices and a unit right-hand side, so a row is satisfied when one positive variable equals one or one negated variable equals one. They also characterized bipolar max--min systems by integer linear inequalities. This result is fundamental for the present study because it establishes that SAT is not merely an analogy for bipolar relational equations. At the same time, it leaves a distinct optimization question open: how should a minimum-cost, genuinely graded bipolar clause system be formulated and solved when the relational coefficients and target clause levels are fuzzy and the composition is an arbitrary continuous t-norm?

A direct application of the general BFRE optimization framework would enumerate admissible relational regions and compare their region-wise optimal points. Such a procedure is exact, but the number of admissible assignments may grow exponentially \cite{ghodousian2025core}. Branch-and-bound itself is not new to BFRE optimization: a modified branch-and-bound method has, for example, been developed for linear models with max-parametric Hamacher BFRE constraints \cite{aliannezhadi2016}. The novelty pursued here is therefore not the generic use of branch-and-bound, but a SAT-specific search state built from clause witnesses and domain intersections under the arbitrary-continuous-t-norm formulation. In a satisfiability setting the general feasible-region representation has additional combinatorial meaning: every row must choose at least one variable whose positive or negative branch attains the prescribed row level. We call such a variable a \emph{clause witness}. This interpretation permits the search to be organized directly over witness choices. It also makes SAT-style propagation and branch-and-bound possible: partial witness assignments shrink variable domains, inconsistent branches can be detected by empty intersections, forced witnesses can be propagated, and the monotone minimum-cost objective supplies valid lower bounds at partial nodes. In particular, rows whose current witness-variable sets are disjoint impose additive cost increments, yielding a packing bound that is stronger than treating one unresolved row at a time.

Motivated by these observations, this paper introduces the \emph{bipolar fuzzy minimum-weight satisfiability} problem (\BFMWSAT) and develops a specialized exact solution procedure. The main contributions are summarized as follows.

\begin{itemize}
    \item A minimum-weight satisfiability problem is formulated directly over fuzzy positive and negative relation matrices, a fuzzy requirement vector, and an arbitrary continuous t-norm. The fuzzy quantities are treated as inputs to the relational reasoning system; their external fuzzification or defuzzification is outside the scope of the model.
    \item Classical minimum-weight CNF satisfiability is proved to be a crisp specialization of the proposed model for \emph{every} continuous t-norm. In this specialization, a Boolean optimum always exists even though the BFRE formulation uses continuous variables in the unit interval.
    \item The feasible set is expressed in SAT-oriented terms using admissible variable domains, effective clause-witness sets, and compatible witness assignments. Only the structural results required for this problem are developed here; the general continuous-BFRE theory is available in \cite{ghodousian2025core}.
    \item An exact branch-and-bound method is developed over clause witnesses. It combines feasibility preprocessing, forced-witness propagation, domain intersection, lower-point closure, a disjoint-witness packing bound, fail-first branching, and cost-guided branch ordering. The method searches the same globally valid solution structure without explicitly generating every complete admissible region in advance.
    \item A detailed numerical example and a reproducible implementation evaluate the method numerically. The experiments include exactness checks against explicit witness enumeration, mixed-integer baselines, preprocessing tests for three continuous t-norms, a focused ablation of the search mechanisms, and scalability tests on crisp and genuinely graded instances.
\end{itemize}

The rest of the paper is organized as follows. Section \ref{sec_problem} introduces minimum-weight CNF satisfiability and the proposed bipolar fuzzy extension. Section \ref{sec_structure} develops the crisp-specialization results and the clause-witness characterization of the graded feasible set. Section \ref{sec_algorithmic} presents preprocessing and the branch-and-bound framework, and Section \ref{sec_algorithm} summarizes the complete method as an algorithm. Section \ref{sec_example} gives a step-by-step numerical example. Section \ref{sec_computational} reports the computational experiments, and Section \ref{sec_conclusion} concludes the paper.

% ==================================== Problem setting ====================================
\section{Problem setting and bipolar fuzzy formulation}\label{sec_problem}
\subsection{Minimum-weight CNF satisfiability}
Let $y_1,\ldots,y_n$ be Boolean variables and let
\[
\mathcal{C}=C_1\wedge C_2\wedge\cdots\wedge C_m
\]
be a formula in conjunctive normal form (CNF). For each clause $C_i$, let $P_i\subseteq\{1,\ldots,n\}$ contain the indices of variables occurring positively and let $N_i\subseteq\{1,\ldots,n\}$ contain the indices occurring negatively. Thus,
\[
C_i=\bigvee_{j\in P_i}y_j\;\vee\!\bigvee_{j\in N_i}\neg y_j.
\]
We use the convention that the maximum over an empty index set is zero. With this convention, an empty clause is correctly represented as unsatisfiable in the crisp relational construction below.

Given nonnegative variable costs $c_j$, the classical minimum-weight satisfiability problem seeks a satisfying Boolean assignment of minimum total cost,
\begin{equation}\label{eq_mwsat}
\min\left\{\sum_{j=1}^{n}c_jy_j:\;y\in\B^n,\; y\models\mathcal{C}\right\}.
\end{equation}
The unit-weight case is Min-Ones SAT. The nonnegative-cost form is the natural one for the minimum-weight interpretation and, more importantly for the present paper, makes the objective coordinatewise non-decreasing. If separate costs are assigned to the true and false states of a variable, the cheaper state can be absorbed into a constant and the variable can be complemented when necessary, yielding an equivalent nonnegative variable-weight representation.

\subsection{Graded bipolar clauses}
We now replace the crisp incidence description by fuzzy relational data. Let $\Iset=\{1,\ldots,m\}$ index clause requirements and $\Jset=\{1,\ldots,n\}$ index decision variables. The input consists of
$
A^+=(a^+_{ij})_{m\times n},~ A^-=(a^-_{ij})_{m\times n},~ b=(b_i)_{m\times1},
$
with all entries in $[0,1]$. The coefficient $a^+_{ij}$ is the positive relational grade between variable $j$ and clause requirement $i$, while $a^-_{ij}$ is its negative or complementary relational grade. The component $b_i$ is the prescribed (target) grade of row $i$. These quantities are assumed to have already been supplied in fuzzy form. Their construction from measurements, linguistic statements, or expert assessments belongs to an external modelling layer and is not part of the optimization problem studied here.

Let $\ph$ be a continuous t-norm and let $x=(x_1,\ldots,x_n)^T\in[0,1]^n$. Variable $x_j$ participates through both $x_j$ and its standard complement $1-x_j$. For row $i$, define
$
g_i(x)=\max_{j\in\Jset}\left\{\max\left\{\ph(a^+_{ij},x_j),\ph(a^-_{ij},1-x_j)\right\}\right\}.
$
The requirement $g_i(x)=b_i$ has two simultaneous consequences: every variable contribution is at most $b_i$, and at least one positive or negative branch attains $b_i$. This ``one branch attains the row level'' property is the clause-witness structure used later by the algorithm.

Throughout this paper, $b_i$ is a \emph{target grade}, not a lower satisfaction threshold. The equality $g_i(x)=b_i$ is intentional because the model inherits bipolar fuzzy relational-equation semantics: contributions above $b_i$ are inadmissible, while at least one contribution must attain $b_i$ exactly. Replacing the equality by $g_i(x)\geq b_i$ would define a different threshold-satisfaction model, with a different feasible-set structure, and is not considered here.

\begin{definition}[Bipolar fuzzy minimum-weight satisfiability]\label{def_bfmwsat}
Given $A^+$, $A^-$, $b$, a continuous t-norm $\ph$, and a nonnegative cost vector $c\in\R_+^n$, the \BFMWSAT\ problem is
\begin{equation}\label{eq_bfmwsat}
\begin{aligned}
\min_{x\in[0,1]^n}~ & c^Tx\\
\text{s.t.}~ & g_i(x)=b_i,~ i\in\Iset.
\end{aligned}
\end{equation}
\end{definition}

Definition \ref{def_bfmwsat} is a satisfiability-oriented specialization of continuous-BFRE-constrained optimization \cite{ghodousian2025core}. The present paper uses only the structural consequences needed for minimum-weight clause reasoning and develops the search procedure directly in those terms.

\begin{remark}\label{rem_semantics}
The word \emph{satisfiability} is used here in a relational sense. The row map $g_i$ is a max--t-norm bipolar relational composition, not the recursive truth evaluation of an arbitrary formula in Lukasiewicz, G\"odel, product, or another fuzzy propositional logic. Accordingly, \BFMWSAT\ should be compared with fuzzy MaxSAT and fuzzy-SAT models such as \cite{halaby2016,lapenta2026,castro2026} as a different graded semantics rather than as a replacement for them.
\end{remark}

% ==================================== Structural properties ====================================
\section{Structural properties and clause witnesses}\label{sec_structure}
\subsection{Classical minimum-weight SAT as a crisp specialization}
Li and Jin \cite{lijin2016} established the fundamental connection between Boolean satisfiability and bipolar max--min equations. The first result below records the corresponding embedding for the present optimization model and observes that the crisp construction does not depend on which t-norm is used.

For a CNF formula $\mathcal{C}$, construct crisp relational matrices by setting
\[
a^+_{ij}=\begin{cases}1,&j\in P_i,\\0,&j\notin P_i,\end{cases}
~
a^-_{ij}=\begin{cases}1,&j\in N_i,\\0,&j\notin N_i,\end{cases}
~ b_i=1.
\]
Every t-norm satisfies the boundary identities $\ph(1,z)=z$ and $\ph(0,z)=0$.

\begin{theorem}[t-norm-invariant crisp embedding]\label{thm_crisp}
For the crisp data above, the clause formula $\mathcal{C}$ and the bipolar relational system have exactly the same Boolean feasible vectors. This equivalence uses only the boundary axioms of a t-norm and therefore holds for every t-norm, not only for continuous ones. In particular, classical minimum-weight SAT is a crisp special case of \BFMWSAT\ for every continuous t-norm admitted by Definition \ref{def_bfmwsat}.
\end{theorem}

\begin{proof}
For a Boolean vector $y$, row $i$ reduces to
\[
\max\left\{\max_{j\in P_i}y_j,\;\max_{j\in N_i}(1-y_j)\right\}=1.
\]
The equality holds if and only if some positive literal has $y_j=1$ or some negative literal has $y_j=0$, which is precisely the truth condition for $C_i$. Applying the argument to every row proves the equivalence. The objective $c^Ty$ is exactly the objective in \eqref{eq_mwsat}.
\end{proof}

The continuous domain may contain additional fractional feasible points in the crisp case, but such points cannot improve the minimum-weight objective.

\begin{theorem}[exactness of the continuous domain in the crisp case]\label{thm_binary}
Suppose that $A^+$ and $A^-$ are the crisp incidence matrices of a CNF formula, $b=\mathbf{1}$, and $c\geq0$. For every feasible $x\in[0,1]^n$ of Problem \eqref{eq_bfmwsat}, there exists a feasible Boolean vector $y\in\B^n$ such that $c^Ty\leq c^Tx$. Consequently,
\[
\min\{c^Tx:x\in[0,1]^n,\;x\text{ satisfies the crisp BFRE rows}\}
=
\min\{c^Ty:y\in\B^n,\;y\models\mathcal{C}\},
\]
and a Boolean global optimum exists whenever the crisp instance is feasible.
\end{theorem}

\begin{proof}
Let $x$ be feasible and consider a component with $0<x_j<1$. In a crisp row, the two possible contributions of variable $j$ are $x_j$ and $1-x_j$, both strictly smaller than one. Hence this fractional variable cannot be responsible for attaining any row right-hand side. Every row is therefore already witnessed by an endpoint-valued literal not involving that fractional component. Set $x_j$ to zero. Existing witnesses remain unchanged, positive contributions of $j$ can only decrease, and negative contributions of $j$ can only increase to at most the row level one; feasibility is preserved. Since $c_j\geq0$, the objective does not increase. Repeating this operation for every fractional component produces a Boolean feasible vector $y$ with $c^Ty\leq c^Tx$.

Thus every continuous feasible point is dominated in objective value by some Boolean feasible point. Conversely, Theorem \ref{thm_crisp} shows that every Boolean satisfying assignment is feasible for the continuous-domain model. The two optimal values are therefore equal. Because the Boolean feasible set is finite and nonempty whenever the continuous model is feasible, a Boolean optimum exists.
\end{proof}

The endpoint-witness idea is already present in the SAT reduction of Li and Jin for bipolar max--min feasibility \cite{lijin2016}. Theorem \ref{thm_binary} uses the same structural observation to establish objective-value exactness for the minimum-weight model and, by Theorem \ref{thm_crisp}, for the crisp embedding under any continuous t-norm. This exactness does not alter worst-case complexity: the crisp feasibility specialization contains the NP-complete consistency problem studied in \cite{lijin2016}.

\subsection{Admissible domains and effective witnesses}
For a fixed pair $(i,j)$, write
\[
q_{ij}(z)=\max\{\ph(a^+_{ij},z),\ph(a^-_{ij},1-z)\},~ z\in[0,1].
\]
Define the row-wise admissibility and activation sets by
\[
I_{ij}=\{z\in[0,1]:q_{ij}(z)\leq b_i\},~
S_{ij}=\{z\in[0,1]:q_{ij}(z)=b_i\}.
\]
The globally admissible domain of variable $j$ and its effective activation set for row $i$ are
\begin{equation}\label{eq_domain}
I_j=\bigcap_{i\in\Iset}I_{ij},
~
S'_{ij}=S_{ij}\cap I_j.
\end{equation}
Thus $I_j$ contains exactly the values of $x_j$ that do not make any row exceed its required grade, while $S'_{ij}$ contains the globally admissible values at which variable $j$ can attain row $i$.

\begin{proposition}[finite interval structure of the scalar sets]\label{prop_intervalstructure}
For a continuous t-norm, each nonempty $I_{ij}$ and $I_j$ is a closed interval. Each $S_{ij}$ and $S'_{ij}$ is either empty or the union of at most two closed intervals. Consequently, all scalar sets used by the clause-witness search are compact whenever nonempty, and their minima exist.
\end{proposition}

\begin{proof}
For fixed $a\in[0,1]$, the map $z\mapsto\ph(a,z)$ is continuous and non-decreasing. Hence the sublevel set $\{z:\ph(a,z)\leq b\}$ is a closed interval beginning at zero, while a nonempty level set $\{z:\ph(a,z)=b\}$ is a closed interval. The map $z\mapsto\ph(a,1-z)$ is continuous and non-increasing, so its sublevel set is a closed interval ending at one and its nonempty level set is again a closed interval. Therefore $I_{ij}$, the intersection of the positive and negative sublevel sets, is a closed interval. Moreover,
\[
S_{ij}=I_{ij}\cap\bigl(\{z:\ph(a^+_{ij},z)=b_i\}\cup\{z:\ph(a^-_{ij},1-z)=b_i\}\bigr),
\]
so it is the union of at most two closed intervals. Intersecting the $I_{ij}$ gives the interval $I_j$, and intersecting $S_{ij}$ with $I_j$ cannot create additional components.
\end{proof}

Proposition \ref{prop_intervalstructure} is the amount of scalar-set structure needed by the algorithm. Closed-form endpoint formulae for common continuous t-norms and the corresponding general BFRE derivations are given in \cite{ghodousian2025core}. In an implementation, a selected t-norm must therefore be accompanied by exact or certified routines for forming these interval sets, testing intersections, and extracting minima. The combinatorial search developed below begins after these scalar operations are available; it does not assume that an arbitrary continuous function is supplied as an unevaluable black box.

\begin{definition}[effective clause-witness set]\label{def_witnessset}
For each row $i$, define
\[
J_i=\{j\in\Jset:S'_{ij}\neq\varnothing\}.
\]
A variable $j\in J_i$ is called an \emph{effective witness} for clause requirement $i$.
\end{definition}

\begin{proposition}[feasibility by clause witnesses]\label{prop_feas}
A vector $x\in[0,1]^n$ is feasible for Problem \eqref{eq_bfmwsat} if and only if
\begin{enumerate}
    \item $x_j\in I_j$ for every $j\in\Jset$; and
    \item for every $i\in\Iset$, there exists at least one $j\in J_i$ such that $x_j\in S'_{ij}$.
\end{enumerate}
\end{proposition}

\begin{proof}
If $x_j\in I_j$ for all $j$, then $q_{ij}(x_j)\leq b_i$ for every pair $(i,j)$. If, in addition, $x_j\in S'_{ij}$ for some $j$, then $q_{ij}(x_j)=b_i$, so the maximum defining row $i$ equals $b_i$. Conversely, if all relational rows hold, every contribution is at most the corresponding $b_i$, which yields $x_j\in I_{ij}$ for every $i$ and hence $x_j\in I_j$. Each row maximum also has at least one attaining index $j$, and that value lies in $S_{ij}\cap I_j=S'_{ij}$.
\end{proof}

Two immediate infeasibility certificates follow: $I_j=\varnothing$ for some variable, or $J_i=\varnothing$ for some row. They are checked before any combinatorial search.

\subsection{Compatible witness assignments and region-wise optima}
A complete witness assignment selects one effective variable for every row. Let $e:\Iset\rightarrow\Jset$ satisfy $e(i)\in J_i$, and let
\[
\Iset_j(e)=\{i\in\Iset:e(i)=j\}.
\]
The values compatible with all rows assigned to variable $j$ are
\begin{equation}\label{eq_de}
D_j(e)=
\begin{cases}
I_j, & \Iset_j(e)=\varnothing,\\[1mm]
\displaystyle\bigcap_{i\in\Iset_j(e)}S'_{ij}, & \Iset_j(e)\neq\varnothing.
\end{cases}
\end{equation}
We call $e$ \emph{compatible} if $D_j(e)\neq\varnothing$ for every $j$. A compatible assignment defines $D(e)=D_1(e)\times\cdots\times D_n(e)$. Although each $S'_{ij}$ has at most two interval components, repeated intersections are most conveniently stored as finite unions of closed intervals; compactness is preserved and $\min D_j(e)$ is always defined for a compatible assignment.

\begin{theorem}[witness representation and region optimum]\label{thm_witnessregions}
The following statements hold for \BFMWSAT.
\begin{enumerate}
    \item The complete feasible set is the union of $D(e)$ over all compatible complete witness assignments $e$.
    \item For any compatible $e$, the vector $x^*(e)$ defined by $x^*_j(e)=\min D_j(e)$ minimizes $c^Tx$ over $D(e)$.
    \item Hence a global optimum is obtained by comparing the region-wise candidates $x^*(e)$ over compatible $e$.
\end{enumerate}
\end{theorem}

\begin{proof}
If $x\in D(e)$, then $x_j\in I_j$ for every variable and $x_{e(i)}\in S'_{i,e(i)}$ for every row; Proposition \ref{prop_feas} gives feasibility. Conversely, if $x$ is feasible, Proposition \ref{prop_feas} allows one witnessing variable $e(i)$ to be selected for every row. For each $j$, the value $x_j$ then belongs to all activation sets assigned to that variable, so $D_j(e)$ is nonempty and $x\in D(e)$. This proves the first statement.

For the second statement, every $x\in D(e)$ satisfies $x_j\geq\min D_j(e)=x^*_j(e)$. Since $c_j\geq0$, summing $c_jx_j\geq c_jx^*_j(e)$ over all coordinates gives $c^Tx\geq c^Tx^*(e)$. The final statement follows immediately from the union representation.
\end{proof}

The number of candidate witness assignments before compatibility is tested is at most $\prod_{i\in\Iset}|J_i|$, which can be exponential. The next section searches this same exact representation without constructing every complete assignment in advance.

% ==================================== Algorithmic framework ====================================
\section{Preprocessing and exact clause-witness search}\label{sec_algorithmic}
\subsection{Safe preprocessing}
The globally admissible domains $I_j$ in \eqref{eq_domain} are computed from the original instance and are retained throughout the search. When a row is removed during preprocessing, only its need for an explicit witness is removed; the admissibility restrictions already incorporated in the $I_j$ are not enlarged. This convention makes the search state simple and avoids repeatedly rebuilding the scalar BFRE system.

The following reductions are exact and have direct SAT-oriented interpretations.

\paragraph{Infeasibility.}
If $I_j=\varnothing$ for some variable or $J_i=\varnothing$ for some row, Proposition \ref{prop_feas} proves infeasibility.

\paragraph{Fixed variable.}
If $I_j=\{k\}$, then every feasible solution has $x_j=k$. Any row $i$ satisfying $k\in S'_{ij}$ is automatically witnessed and can be removed from the active-row set. The variable itself is kept with the singleton domain so that its objective contribution remains explicit.

\paragraph{Singleton witness.}
If $J_i=\{j\}$, every feasible solution must satisfy $x_j\in S'_{ij}$. In particular, if $S'_{ij}=\{k\}$, then $x_j=k$ is forced. The more general non-singleton case is handled by the propagation step during search.

\paragraph{Witness-dominance redundancy.}
If two active rows $r$ and $i$ satisfy $S'_{rj}\subseteq S'_{ij}$ for every $j\in\Jset$, then row $i$ is redundant as long as row $r$ is retained: any witness that attains row $r$ also attains row $i$. Likewise, if $S'_{ij}=I_j$ for some $j$, row $i$ is automatically witnessed by every globally admissible value of variable $j$ and can be removed. These are the clause-witness forms of two general BFRE reduction rules in \cite{ghodousian2025core}.

These reductions may be applied repeatedly. More aggressive equivalent transformations from \cite{ghodousian2025core} can also be used, but they are optional; the correctness proof of the search below relies only on exact scalar sets and on reductions whose equivalence has been certified.

\subsection{Partial witness assignments and propagation}
Let $\mathcal{A}\subseteq\Iset$ denote the rows that remain active after preprocessing. A search node is a partial witness assignment $p$ defined on $R(p)\subseteq\mathcal{A}$. Its current domain for variable $j$ is
\begin{equation}\label{eq_partialdomain}
D_j(p)=I_j\cap\bigcap_{\substack{i\in R(p)\\p(i)=j}}S'_{ij},
\end{equation}
where the empty intersection leaves $D_j(p)=I_j$. For an unassigned active row $i\in\mathcal{A}\setminus R(p)$, the currently compatible witnesses are
\begin{equation}\label{eq_currentwitness}
J_i(p)=\{j\in J_i:D_j(p)\cap S'_{ij}\neq\varnothing\}.
\end{equation}

\begin{lemma}[partial-node invariants]\label{lem_node}
Let $e$ be a compatible complete witness assignment extending a partial assignment $p$. Then $D_j(e)\subseteq D_j(p)$ for every $j$. Consequently:
\begin{enumerate}
    \item if some $D_j(p)$ is empty, $p$ has no compatible completion;
    \item if $J_i(p)=\varnothing$ for an unassigned active row $i$, $p$ has no compatible completion; and
    \item if $J_i(p)=\{j\}$, every compatible completion of $p$ must assign row $i$ to $j$.
\end{enumerate}
\end{lemma}

\begin{proof}
An extension $e$ can only add activation-set intersections to those already imposed by $p$, which gives $D_j(e)\subseteq D_j(p)$. The first statement follows immediately. For the second, if a completion assigned row $i$ to some $j$, then $D_j(e)\subseteq D_j(p)\cap S'_{ij}$ would be nonempty, contradicting $J_i(p)=\varnothing$. The same argument shows that when only one $j$ remains in $J_i(p)$, every completion must select that witness.
\end{proof}

Lemma \ref{lem_node} justifies forced-witness propagation. Whenever an unassigned row has $|J_i(p)|=1$, its unique witness is assigned, the corresponding domain is intersected with $S'_{ij}$, all affected $J_r(p)$ are recomputed, and the process is repeated until no singleton row remains or infeasibility is detected.

\subsection{Objective bounds, node closure, and branching}
For a nonempty node $p$, define its coordinatewise lower point by $\underline{x}_j(p)=\min D_j(p)$. The basic domain lower bound is
\[
\operatorname{LB}_0(p)=\sum_{j=1}^{n}c_j\underline{x}_j(p).
\]
Every descendant can only intersect the current domains further, so $\operatorname{LB}_0$ is non-decreasing along every root-to-leaf path.

For an unassigned active row $i$, define its least additional witness cost by
\[
\delta_i(p)=\min_{j\in J_i(p)}c_j\left[\min\bigl(D_j(p)\cap S'_{ij}\bigr)-\min D_j(p)\right].
\]
The bracketed term is nonnegative. Hence $\operatorname{LB}_0(p)+\max_i\delta_i(p)$ is a valid single-row strengthening, because every completion must eventually witness each unresolved row. It does not, however, exploit the fact that some rows cannot share a witnessing variable.

A set $K\subseteq\mathcal{A}\setminus R(p)$ is called a \emph{disjoint-witness packing} at node $p$ if
\[
J_r(p)\cap J_i(p)=\varnothing,~ r,i\in K,\ r\neq i.
\]
For such a packing, define
\begin{equation}\label{eq_lb}
\operatorname{LB}_{K}(p)=\operatorname{LB}_0(p)+\sum_{i\in K}\delta_i(p).
\end{equation}

\begin{proposition}[disjoint-witness packing bound]\label{prop_packing}
For every search node $p$ with nonempty current domains and every disjoint-witness packing $K$, $\operatorname{LB}_{K}(p)$ is a lower bound on the objective value of every feasible completion of $p$.
\end{proposition}

\begin{proof}
Consider any compatible complete witness assignment $e$ extending $p$. For each $i\in K$, the selected witness $j_i=e(i)$ belongs to $J_i(p)$. Since the witness sets of the rows in $K$ are pairwise disjoint, the variables $j_i$ are distinct. The completed value of $x_{j_i}$ belongs to $D_{j_i}(p)\cap S'_{i j_i}$, so relative to the basic lower point its objective contribution increases by at least
$c_{j_i}[\min(D_{j_i}(p)\cap S'_{i j_i})-\min D_{j_i}(p)]\geq\delta_i(p)$. These increases occur on distinct variables and can therefore be added without double counting. All remaining coordinate increases are nonnegative because $c\geq0$. Thus every feasible completion has objective value at least $\operatorname{LB}_{K}(p)$.
\end{proof}

The use of disjoint structures to strengthen branch-and-bound lower bounds has an important precedent in exact MaxSAT. Li, Many\`a, and Planes \cite{li2006disjoint} interpreted several MaxSAT lower-bound procedures as searches for disjoint inconsistent subformulas, whose contributions can be added without double counting. The present construction is different: it does not search for inconsistent Boolean subformulas. Instead, it packs unresolved BF-MWSAT rows whose \emph{current witness-variable sets} are disjoint and adds their graded minimum-cost activation increments. Thus, the common principle is additivity under disjointness, while the objects being packed and the cost quantities are specific to the bipolar fuzzy relational model.

The strongest bound of this form would choose a maximum-weight collection of pairwise-disjoint current witness sets. Solving that auxiliary packing problem exactly at every search node would introduce unnecessary combinatorial overhead. We therefore use a greedy packing $K_g(p)$: process the unresolved rows in non-increasing order of $\delta_i(p)$ and insert a row whenever its current witness set is disjoint from those already selected. The first selected row has maximum $\delta_i(p)$, and every additional selected row contributes a nonnegative increment. Consequently, the greedy packing bound
$\operatorname{LB}_g(p)=\operatorname{LB}_{K_g(p)}(p)$ is valid and is never weaker than the single-row bound.

The lower point also provides an exact node-closure test.
\begin{proposition}[lower-point closure]\label{prop_bound}
For every search node $p$ with nonempty current domains:
\begin{enumerate}
    \item $\operatorname{LB}_0(p)$ is a lower bound on every feasible completion, and if $q$ extends $p$, then $\operatorname{LB}_0(q)\geq\operatorname{LB}_0(p)$;
    \item if $\underline{x}(p)$ is feasible for the original \BFMWSAT\ instance, then it is an optimal feasible completion of node $p$, with objective value $\operatorname{LB}_0(p)$, and the node can be fathomed immediately.
\end{enumerate}
\end{proposition}

\begin{proof}
Every feasible completion $x$ satisfies $x_j\in D_j(p)$, hence $x_j\geq\min D_j(p)$ and $c^Tx\geq\operatorname{LB}_0(p)$. If $q$ extends $p$, Lemma \ref{lem_node} gives $D_j(q)\subseteq D_j(p)$ and therefore $\min D_j(q)\geq\min D_j(p)$. Finally, if $\underline{x}(p)$ is itself feasible, it attains the basic lower bound; no feasible completion of the node can have smaller cost.
\end{proof}

If an incumbent of value $z^*$ is available, a node is pruned whenever $\operatorname{LB}_g(p)\geq z^*$. Before branching, $\underline{x}(p)$ is tested against the original relational rows so that Proposition \ref{prop_bound} can close a node as soon as its lower bound is attainable.

When branching is necessary, choose an unassigned row with the smallest current witness set,
\[
i^*\in\arg\min_{i\in\mathcal{A}\setminus R(p)}|J_i(p)|.
\]
Among ties, a row with larger $\delta_i(p)$ is preferred. For each $j\in J_{i^*}(p)$, the child assigns $p(i^*)=j$ and intersects $D_j(p)$ with $S'_{i^*j}$. Children are explored in nondecreasing greedy packing-bound order. These choices affect search efficiency but not correctness.

\begin{theorem}[correctness of clause-witness branch-and-bound]\label{thm_bb}
Assume that the scalar sets in \eqref{eq_domain} and all subsequent set operations are computed exactly, that $c\geq0$, and that preprocessing uses only equivalence-preserving reductions. Then the branch-and-bound procedure based on \eqref{eq_partialdomain}--\eqref{eq_lb}, repeated forced-witness propagation, the greedy disjoint-witness packing bound, lower-point closure, and complete branching over $J_i(p)$ terminates and returns a globally optimal solution of \BFMWSAT\ whenever the problem is feasible; otherwise it returns infeasibility.
\end{theorem}

\begin{proof}
Only finitely many active rows can be assigned. Every forced or branching step assigns at least one previously unassigned row, so propagation terminates and every root-to-leaf path is finite. Since each row has finitely many effective witnesses, the complete search tree is finite.

Consider any compatible complete witness assignment $e$. At every partial node $p$ on the path consistent with $e$, Lemma \ref{lem_node} guarantees that the witness $e(i)$ of an unassigned row remains in $J_i(p)$; hence complete branching contains a child consistent with $e$. Empty-domain and empty-witness pruning cannot remove such a path. Singleton propagation is also safe because Lemma \ref{lem_node} shows that its assignment is necessary for every completion of the current node.

By Proposition \ref{prop_packing}, pruning with the greedy packing bound removes only nodes whose feasible completions cannot improve the incumbent. If the lower point of a partial node is feasible, Proposition \ref{prop_bound} proves that it is already the optimum over that node, so fathoming it loses no improving solution. If the search reaches a complete compatible assignment, Theorem \ref{thm_witnessregions} shows that the coordinatewise domain minima give the exact optimum over its witness region.

Finally, Theorem \ref{thm_witnessregions} states that every feasible solution belongs to at least one compatible complete witness region. Therefore a globally optimal region either is explicitly reached or is safely fathomed/pruned by a valid bound after an incumbent of no larger value has been found. The best incumbent returned by the finite search is consequently globally optimal. If no compatible region survives, the instance is infeasible.
\end{proof}

\begin{remark}
The worst-case search remains exponential, which is unavoidable in view of the NP-complete crisp feasibility specialization \cite{lijin2016}. The algorithmic contribution is instead the organization of this search around clause witnesses: incompatible domain intersections are rejected before completion, singleton witnesses propagate, disjoint witness sets yield additive objective increments, and a feasible lower point closes a node without enumerating all witness completions below it.
\end{remark}

% ==================================== Algorithm ====================================
\section{Solution algorithm}\label{sec_algorithm}
Algorithm \ref{alg_bfmwsat} summarizes the complete method. Proposition \ref{prop_intervalstructure} reduces the scalar stage to interval-set calculations, while Sections \ref{sec_structure}--\ref{sec_algorithmic} justify the combinatorial search. The origin of the fuzzy inputs is immaterial once $A^+$, $A^-$, $b$, and the chosen t-norm have been supplied.

\begin{algorithm}[H]
\footnotesize
\caption{Exact clause-witness algorithm for \BFMWSAT}\label{alg_bfmwsat}
\KwIn{$A^+$, $A^-$, $b$, $c\geq0$, a continuous t-norm $\ph$, and exact scalar-set routines.}
\KwOut{A global optimum $(x^*,z^*)$, or infeasibility.}
Compute $I_{ij}$, $S_{ij}$, $I_j$, $S'_{ij}$, and $J_i$\;
\If{some $I_j=\varnothing$ or $J_i=\varnothing$}{\Return infeasible\;}
Set the active-row set $\mathcal{A}\leftarrow\Iset$ and apply the safe preprocessing rules of Section \ref{sec_algorithmic}; keep all $I_j$ fixed\;
Initialize the root node $p$ with $R(p)=\varnothing$ and $D_j(p)=I_j$; set $z^*\leftarrow+\infty$, $x^*\leftarrow\varnothing$\;
Insert the root into the search list\;
\While{the search list is nonempty}{
  Remove a node $p$\;
  Recompute $J_i(p)$ and repeatedly propagate every unassigned row with $|J_i(p)|=1$ until closure\;
  \If{some current $D_j(p)$ or $J_i(p)$ is empty}{continue;}
  Set $\underline{x}_j\leftarrow\min D_j(p)$ and $z_0\leftarrow\sum_j c_j\underline{x}_j$\;
  \If{$z_0\geq z^*$}{continue;}
  \If{$\underline{x}$ satisfies every original row}{
      Set $(x^*,z^*)\leftarrow(\underline{x},z_0)$; continue\;
  }
  Compute $\delta_i(p)$ and current $J_i(p)$ for the unassigned active rows\;
  Build the greedy disjoint-witness packing $K_g(p)$ and set $\operatorname{LB}_g(p)$ from \eqref{eq_lb}\;
  \If{$\operatorname{LB}_g(p)\geq z^*$}{continue;}
  Choose an unassigned $i^*$ minimizing $|J_i(p)|$, breaking ties by larger $\delta_i(p)$\;
  Generate one child for every $j\in J_{i^*}(p)$ by assigning $p(i^*)=j$ and replacing $D_j(p)$ with $D_j(p)\cap S'_{i^*j}$\;
  Insert the children in nondecreasing greedy packing-bound order\;
}
\Return{$(x^*,z^*)$ if $x^*\neq\varnothing$; otherwise infeasible}\;
\end{algorithm}

A depth-first implementation with best-child-first ordering keeps memory small, whereas a priority queue keyed by the lower bound gives a best-bound strategy. Theorem \ref{thm_bb} applies to either policy provided every unpruned node is eventually processed. The exactness claim concerns the mathematical set operations; a numerical implementation should use closed-form endpoints or certified tolerances appropriate to the selected t-norm rather than silently treating floating-point equality as exact.

% ==================================== Numerical example ====================================
\section{Detailed numerical example}\label{sec_example}
This section illustrates the method numerically. The example is deliberately small enough to show every relevant step while still exhibiting variable fixing, row elimination, branching, and lower-bound pruning. The minimum t-norm $\ph(u,v)=\min\{u,v\}$ is used, but the algorithm itself is not restricted to this choice.

Consider five graded clause requirements and six variables. Let
\begin{equation}\label{eq_exampledata}
A^+=\begin{bmatrix}
0.80&0.10&0.10&0.10&0.10&0.90\\
0.10&0.80&0.10&0.10&0.10&0.10\\
0.10&0.10&0.10&0.10&0.10&0.10\\
0.10&0.10&0.10&0.90&0.80&0.85\\
0.90&0.10&0.10&0.10&0.10&0.10
\end{bmatrix},~
A^-=\begin{bmatrix}
0.10&0.10&0.10&0.85&0.10&0.10\\
0.10&0.10&0.10&0.10&0.10&0.10\\
0.85&0.90&0.10&0.10&0.80&0.10\\
0.10&0.10&0.10&0.10&0.10&0.10\\
0.10&0.10&0.10&0.90&0.80&0.85
\end{bmatrix},
\end{equation}
with
\[
b=(0.60,0.40,0.60,0.60,0.60)^T,
~
c=(5,2,4,3,6,1)^T.
\]
The objective is to minimize $c^Tx$.

For the minimum t-norm, the pairwise admissibility and activation sets are easy to calculate from the threshold structure of $\min$. Intersecting the row-wise admissibility sets gives
\begin{equation}\label{eq_exampledomains}
I_1=[0.4,0.6],~ I_2=\{0.4\},~ I_3=[0,1],~
I_4=I_5=I_6=[0.4,0.6].
\end{equation}
The nonempty effective activation sets are listed in Table \ref{tab_activation}; all omitted entries are empty.

\begin{table}[!ht]
\centering
\caption{Nonempty effective clause-witness sets for the numerical example.}\label{tab_activation}
\begin{tabular}{c|cccccc}
\hline
 & $x_1$ & $x_2$ & $x_3$ & $x_4$ & $x_5$ & $x_6$\\
\hline
row 1 & $\{0.6\}$ & -- & -- & $\{0.4\}$ & -- & $\{0.6\}$\\
row 2 & -- & $\{0.4\}$ & -- & -- & -- & --\\
row 3 & $\{0.4\}$ & $\{0.4\}$ & -- & -- & $\{0.4\}$ & --\\
row 4 & -- & -- & -- & $\{0.6\}$ & $\{0.6\}$ & $\{0.6\}$\\
row 5 & $\{0.6\}$ & -- & -- & $\{0.4\}$ & $\{0.4\}$ & $\{0.4\}$\\
\hline
\end{tabular}
\end{table}

The first preprocessing effect follows immediately from \eqref{eq_exampledomains}: $I_2=\{0.4\}$, so every feasible solution has $x_2=0.4$. At this value, row 2 is activated by its positive branch and row 3 is activated by its negative branch. Hence both rows are already satisfied and can be removed from the remaining witness search. Variable $x_2$ is retained at its fixed value when the full solution is reconstructed.

After this reduction, the unresolved rows are 1, 4, and 5. Their effective witness choices are
\[
J_1=\{1,4,6\},~
J_4=\{4,5,6\},~
J_5=\{1,4,5,6\}.
\]
Before any of these rows is assigned, the coordinatewise minimum of the current domains is
\[
\underline{x}=(0.4,0.4,0,0.4,0.4,0.4)^T.
\]
Its objective value is
\[
5(0.4)+2(0.4)+4(0)+3(0.4)+6(0.4)+1(0.4)=6.8,
\]
which is the basic domain bound $\operatorname{LB}_0$ at the root. Rows 1 and 5 are already witnessed at this lower point, whereas row 4 is not. For row 4 the three possible witnesses require $x_4=0.6$, $x_5=0.6$, or $x_6=0.6$, with incremental costs
\[
3(0.6-0.4)=0.6,~
6(0.6-0.4)=1.2,~
1(0.6-0.4)=0.2.
\]
Hence $\delta_4=0.2$, while the already witnessed rows have zero incremental requirement. Moreover, $J_1$, $J_4$, and $J_5$ intersect pairwise, so a disjoint-witness packing can contain at most one of these rows. The greedy packing therefore selects row 4 and the bound in \eqref{eq_lb} is $6.8+0.2=7.0$. Row 4 is tied with row 1 in witness-set cardinality but has the larger $\delta_i$, so the tie-breaking rule selects row 4. The cost-guided branch order explores $x_6=0.6$ first.

At the resulting node the coordinatewise lower point has $x_6=0.6$ and all other variables at their current minima. Row 1 is witnessed by the negative branch associated with $x_4=0.4$, and row 5 is witnessed by the same value. Thus the node lower point is already feasible. Proposition \ref{prop_bound} closes the node immediately and yields
\begin{equation}\label{eq_examplestar}
x^*=(0.4,0.4,0,0.4,0.4,0.6)^T,
~ c^Tx^*=7.0.
\end{equation}
This solution becomes the incumbent and is the exact optimum of the first child node.

The two remaining branches of row 4 can now be bounded immediately. Choosing $x_4=0.6$ gives a lower bound of $6.8+0.6=7.4$, while choosing $x_5=0.6$ gives $6.8+1.2=8.0$. Both bounds exceed the incumbent value $7.0$, so both branches are pruned. No additional complete witness assignments need to be generated.

For completeness, direct substitution verifies the final vector. Row 1 attains $0.6$ through the negative branch associated with $x_4=0.4$; row 2 attains $0.4$ through the positive branch of $x_2=0.4$; row 3 attains $0.6$ through the negative branch of the same variable; row 4 attains $0.6$ through the positive branch of $x_6=0.6$; and row 5 attains $0.6$ through the negative branch associated with $x_4=0.4$. All other contributions stay below the corresponding right-hand side. Thus \eqref{eq_examplestar} is feasible, and the branch-and-bound lower bounds prove that no feasible solution has smaller cost.

This example also illustrates why the clause-witness view is computationally useful. A generic complete-region construction would first describe all compatible ways of assigning witnesses to the remaining rows. The proposed search instead finds a low-cost complete witness assignment early and uses its objective value to discard alternative witness choices before their descendants are generated.

\section{Computational experiments}\label{sec_computational}

The computational study has four purposes: to verify the implementation against independent exact formulations, to quantify the effect of BFRE preprocessing, to isolate the contribution of the disjoint-witness packing bound and other search mechanisms, and to examine the scalability of the proposed approach on crisp and genuinely graded instances. The implementation package accompanying the paper reproduces the instance generators, experimental settings, and reported tables. The generic mixed-integer baseline was solved using the HiGHS optimization engine through the \texttt{scipy.optimize.milp} interface \cite{huangfu2018}. Fixed random seeds were used throughout to ensure reproducibility.

\subsection{Exactness checks and preprocessing}

Two independent checks were performed before studying search performance. First, small graded instances with $n=8$, $m=10$, and three intended witnesses per row were generated for the minimum, product, and Lukasiewicz t-norms. For five instances of each t-norm, the proposed implementation was compared with explicit enumeration of complete witness assignments. All 15 objective values agreed. The enumeration baseline visited a median of $88{,}573$ search nodes and generated $59{,}049$ complete witness assignments, whereas the proposed implementation visited a median of one node and therefore typically closed these deliberately reduction-friendly validation instances at the root. This experiment is intended as a correctness check rather than as a runtime comparison.

A second family with $n=16$ and $m=24$ was used to isolate preprocessing. Five instances were generated for each of the three t-norms. Table \ref{tab_preprocessing} reports medians. The exact MILP baseline agreed with the proposed implementation on all 15 instances. The generated family intentionally contains fixed-variable and redundant-row structure; preprocessing therefore removes all active rows in the median case, while disabling preprocessing requires a median of 73 search nodes.

\begin{table}[!ht]
	\centering
	\caption{Effect of preprocessing on graded validation instances ($n=16$, $m=24$, five instances per t-norm).}
	\label{tab_preprocessing}
	\begin{tabular}{lrrrr}
		\hline
		T-norm & Fixed vars. & Active rows after pre. & Nodes with pre. & Nodes without pre.\\
		\hline
		Minimum & 12 & 0 & 1 & 73\\
		Product & 12 & 0 & 1 & 73\\
		Lukasiewicz & 11 & 0 & 1 & 73\\
		\hline
	\end{tabular}
\end{table}

\subsection{Ablation of the clause-witness search}

The main ablation uses eight genuinely graded minimum-t-norm instances with $n=12$, $m=50$, and three candidate witnesses per row. Each variant receives the same two-second time limit. The \emph{full packing} variant is Algorithm \ref{alg_bfmwsat}. The \emph{single-row bound} keeps only $\operatorname{LB}_0+\max_i\delta_i$; \emph{domain bound only} uses $\operatorname{LB}_0$ without a witness increment; the remaining variants disable lower-point closure or forced-witness propagation, respectively.

\begin{table}[!ht]
	\centering
	\caption{Ablation on eight graded $12\times50$ instances. Times are wall-clock medians over all runs.}
	\label{tab_ablation}
	\begin{tabular}{lrrrr}
		\hline
		Variant & Certified & Median nodes & Median time (s) & Median forced\\
		\hline
		Full packing & 8/8 & 43.5 & 0.081 & 6.0\\
		Single-row bound & 8/8 & 79.0 & 0.162 & 11.0\\
		Domain bound only & 5/8 & 293.5 & 0.271 & 72.5\\
		No lower-point closure & 8/8 & 139.0 & 0.142 & 8.5\\
		No propagation & 8/8 & 44.5 & 0.091 & 0.0\\
		\hline
	\end{tabular}
\end{table}

Relative to the single-row bound, the disjoint-witness packing bound reduces the median number of explored nodes by about 45\% and the median runtime by about 50\% on this ablation set. The principle is related to lower-bound constructions in MaxSAT, where disjoint inconsistent structures are identified to obtain additive lower bounds without double counting \cite{li2006disjoint}. However, the present bound is defined over unresolved BF-MWSAT rows and their disjoint witness-variable sets, rather than over inconsistent Boolean subformulas. Removing all witness-aware increments is substantially more damaging: three of the eight runs reach the time limit without certifying optimality. Lower-point closure also has a clear effect. Forced-witness propagation is less influential on this particular family, although it remains an exact and inexpensive reduction and becomes more active in larger instances.

\subsection{Generated crisp scalability}

For the crisp specialization, satisfiable planted 3-CNF instances were generated at approximately $4.55n$ clauses, with independent integer variable costs in $\{1,\ldots,10\}$. The density is comparable to the small uniform Random-3-SAT family \texttt{uf20-91} reported by SATLIB \cite{hoos2000satlib}, although the present planted instances are generated independently and are not SATLIB instances. The BFRE implementation is compared with a binary MILP formulation of minimum-weight SAT. Table \ref{tab_crisp_scale} shows that the optimum agrees with MILP on every generated instance for which an incumbent is reported. One $20\times91$ case reaches the three-second BFRE limit after finding the MILP-optimal incumbent but before completing the proof of optimality.

\begin{table}[!ht]
	\centering
	\caption{Crisp planted 3-CNF scalability with a three-second BFRE time limit.}\label{tab_crisp_scale}
	\begin{tabular}{rrrrrr}
		\hline
		$n$ & $m$ & Instances & Certified & Median nodes & Median BFRE / MILP time (s)\\
		\hline
		12 & 55 & 5 & 5 & 16 & 0.055 / 0.010\\
		20 & 91 & 5 & 4 & 108 & 0.461 / 0.075\\
		28 & 127 & 4 & 4 & 199 & 1.409 / 0.154\\
		\hline
	\end{tabular}
\end{table}

\subsection{External SATLIB check}

To complement the generated crisp instances with formulas not constructed by the study, we also tested the first five satisfiable formulas \texttt{uf20-01.cnf}--\texttt{uf20-05.cnf} from the SATLIB \texttt{uf20-91} family \cite{hoos2000satlib}. Each formula has 20 variables and 91 clauses. The benchmark files were used in their original DIMACS clause form and converted to the crisp BF-MWSAT representation of Theorem \ref{thm_crisp}. Two objective schemes were considered. The \emph{unit} scheme sets all costs to one and therefore measures Min-Ones SAT. The \emph{weighted} scheme uses a fixed integer cost vector in $\{1,\ldots,10\}^{20}$ generated once with a fixed seed; the same vector is used for all five formulas. Each BF-MWSAT run has a ten-second limit and is checked against the binary MILP baseline.

\begin{table}[!ht]
	\centering
	\caption{External SATLIB \texttt{uf20-91} check on the first five satisfiable formulas. Times are wall-clock medians over the five formulas.}\label{tab_satlib_external}
	\begin{tabular}{lrrrrr}
		\hline
		Cost scheme & Instances & Certified & Optimal incumbents & Median nodes & Median BFRE / MILP time (s)\\
		\hline
		Unit & 5 & 5 & 5 & 66 & 0.139 / 0.018\\
		Weighted & 5 & 4 & 5 & 86 & 0.159 / 0.018\\
		\hline
	\end{tabular}
\end{table}

All five unit-cost cases are certified, and all five objectives agree with MILP. Under the fixed weighted costs, four cases are certified within ten seconds; on \texttt{uf20-02.cnf}, the BF-MWSAT search reaches the time limit after finding the same objective value as MILP but before completing the optimality proof. This small external experiment is not intended as a broad SAT competition study. Its role is to verify that the crisp embedding and optimization procedure behave consistently on standard formulas that were not generated to favor the proposed structure.

\subsection{Genuinely graded scalability}

A harder graded minimum-t-norm family was then generated. Each row has three bipolar candidate witnesses, fuzzy requirement levels selected from $\{0.4,0.5,0.6\}$, and a planted feasible point. Table \ref{tab_graded_scale} reports the same three-second limit. All instances at $12\times50$ and $16\times70$ are certified by the proposed implementation. At $20\times91$ and $24\times110$, the prototype increasingly reaches the time limit; the incumbent matches the MILP optimum in three of four and one of three cases, respectively, but these matches are not counted as certificates. The compiled HiGHS baseline is faster at these sizes. Thus, the computational contribution is the structural exact search and its reductions, not a claim of state-of-the-art performance against a mature compiled MIP solver.

\begin{table}[!ht]
	\centering
	\caption{Genuinely graded minimum-t-norm scalability with a three-second BFRE time limit.}\label{tab_graded_scale}
	\begin{tabular}{rrrrrrr}
		\hline
		$n$ & $m$ & Instances & Certified & Optimal incumbents & Median nodes & Median BFRE / MILP time (s)\\
		\hline
		12 & 50 & 5 & 5 & 5 & 44 & 0.107 / 0.031\\
		16 & 70 & 5 & 5 & 5 & 69 & 0.181 / 0.036\\
		20 & 91 & 4 & 0 & 3 & 926 & 3.002 / 0.322\\
		24 & 110 & 3 & 0 & 1 & 703 & 3.002 / 1.114\\
		\hline
	\end{tabular}
\end{table}

The experiments support four conclusions. First, the interval-set and witness implementation reproduces independent exact baselines on the validation instances. Second, BFRE preprocessing and lower-point closure can collapse substantial portions of the search before complete witness assignments are formed. Third, the new packing bound provides a measurable improvement over the single-row bound while preserving exactness. Fourth, the SATLIB check confirms the crisp formulation on external formulas not generated by the study. The larger graded cases also show the present limitation clearly: further implementation engineering, stronger bounds, decomposition, and broader external benchmark suites are natural next steps rather than claims already established by the current prototype.

% ==================================== Conclusion ====================================
\section{Conclusion}

This paper introduced the bipolar fuzzy minimum-weight satisfiability problem, a continuous optimization model that extends classical minimum-weight CNF satisfiability through bipolar fuzzy relational equalities. The proposed formulation preserves the classical problem as a crisp special case while providing a graded framework in which clause satisfaction is represented through continuous-t-norm interactions between variables and clause witnesses.

The structure of the proposed model was characterized through admissible domains, effective clause witnesses, and compatible witness regions. These properties enabled the development of an exact branch-and-bound algorithm that combines BFRE-based preprocessing, witness propagation, lower-point closure, and a disjoint-witness packing bound. The proposed bound exploits the additive contribution of independent unresolved constraints and provides a stronger pruning mechanism than a single-row lower bound while preserving global optimality.

The computational study verified the correctness of the implementation through comparisons with explicit witness enumeration and mixed-integer optimization formulations. The experiments demonstrated that preprocessing, lower-point closure, and the packing bound can substantially reduce the search effort. The external SATLIB evaluation further confirmed that the crisp specialization behaves consistently on benchmark formulas that were not generated specifically for the proposed framework.

The current results establish a structural and algorithmic foundation for solving bipolar fuzzy minimum-weight satisfiability problems. Future research may investigate more advanced decomposition strategies, stronger lower bounds, parallel search procedures, and larger benchmark collections to improve scalability while preserving the exact nature of the method.

% ==================================== References ====================================
\printbibliography
\end{document}